\documentclass{article}
\usepackage[cp1251]{inputenc}
\usepackage[russian]{babel}
\usepackage{amssymb,amsfonts,amsmath,amsthm,amscd}
\usepackage{graphicx}
\usepackage{enumerate}
\usepackage{soul}
\usepackage[matrix,arrow,curve]{xy}
\usepackage{longtable}
\usepackage{array}
\usepackage{bm}
\usepackage{titling}

\makeatletter

\def\@maketitle{%
  \newpage
  \vskip0.5em%
  UDK \udk%
  \vskip0.5em%
  \vskip1em%
  \begin{center}%
  \let\footnote\thanks%
   {\Large\@author\par}%
   \vskip1.5em%
   {\bf\LARGE\@title\par}%
   \vskip1em%
  \end{center}%
  \par
  \vskip1.5em}

\def\@title{\@latex@warning@no@line{No \noexpand\title given}}

\renewcommand{\thefootnote}{\arabic{footnote}}

\renewcommand{\thanksmarkseries}[1]{%
  \def\@bsmarkseries{\renewcommand{\thefootnote}{\arabic{footnote}}}}
\thanksmarkseries{arabic}

\title{Topological and homological properties\\
of the orbit space\\
of a~simple three-dimensional\\
compact linear Lie group}
\author{Styrt O.\,G.\thanks{Russia, MIPT, oleg\_styrt@mail.ru}}
\newcommand{\udk}{512.815.1+512.815.6+512.816.1+512.816.2}

\newdimen\defskip
\newtheorem{theorem}{Theorem}
\newtheorem{lemma}{Lemma}
\newtheorem{prop}{Proposition}
\newtheorem{stm}{Statement}
\newtheorem{imp}{Corollary}

\newcommand{\parr}{\par\addvspace{\defskip}}
\newcommand{\deff}[2]{\newenvironment{#1}{\parr\textbf{#2.}}{\parr}}
\deff{df}{Definition}
\deff{note}{Remark}
\deff{denote}{Notation}
\deff{denotes}{Notations}
\deff{hint}{Hint}
\deff{answer}{Answer\:}

\def\@thm#1#2#3{%
  \ifhmode\unskip\unskip\par\fi
  \normalfont
  \trivlist
  \let\thmheadnl\relax
  \let\thm@swap\@gobble
  \thm@notefont{\bfseries\upshape}%
  \thm@headpunct{.}
  \thm@headsep 5\p@ plus\p@ minus\p@\relax
  \thm@space@setup
  #1
  \@topsep \thm@preskip               
  \@topsepadd \thm@postskip           
  \def\@tempa{#2}\ifx\@empty\@tempa
    \def\@tempa{\@oparg{\@begintheorem{#3}{}}[]}%
  \else
    \refstepcounter{#2}%
    \def\@tempa{\@oparg{\@begintheorem{#3}{\csname the#2\endcsname}}[]}%
  \fi
  \@tempa
}

\def\proof{\relax\ifmmode \blacktriangleleft
           \mskip 36mu plus 0mu minus 18mu
           \else
           \par\addvspace\medskipamount\noindent$\blacktriangleleft$\ \fi}
\def\endproof{\ifmmode \mskip 36mu plus 0mu
                        minus 18mu\blacktriangleright
              \else
              \unskip~$\blacktriangleright$\par
              \vskip\medskipamount\fi}

\newcommand{\Ss}{\textup{\S\,}}
\newcommand{\Sss}{\textup{\S\S\,}}

\newcommand*{\clei}{\nobreak\hskip\z@skip}
\DeclareRobustCommand*{\ti}{~\textemdash{} }
\DeclareRobustCommand*{\dh}{\clei\hbox{-}\clei}

\renewcommand{\:}{\textup{:}}
\renewcommand{\~}{\textup{;}}

\makeatother

\newcommand*{\bw}[1]{#1\nobreak\discretionary{}{\hbox{$\mathsurround=0pt #1$}}{}}
\newcommand{\sco}{,\ldots,}

\newcommand{\sge}{\bw\ge\ldots\bw\ge}
\newcommand{\sop}{\bw\oplus\ldots\bw\oplus}

\newcommand{\sti}{\bw\times\ldots\bw\times}

\newcommand{\ha}[1]{\left\langle#1\right\rangle}

\newcommand{\br}[1]{\bigl(#1\bigr)}

\newcommand{\ter}[1]{\textup{(}#1\textup{)}}

\newcommand{\hn}[1]{\left\|#1\right\|}

\newcommand{\bs}[1]{\bigl[#1\bigr]}

\newcommand{\bc}[1]{\bigl\{#1\bigr\}}

\newcommand{\mbb}{\mathbb}
\newcommand{\mbf}{\mathbf}
\newcommand{\mcl}{\mathcal}
\newcommand{\mfr}{\mathfrak}

\newcommand{\R}{\mbb{R}}

\newcommand{\Z}{\mbb{Z}}
\newcommand{\N}{\mbb{N}}
\newcommand{\T}{\mbb{T}}

\newcommand{\Cbb}{\mbb{C}}
\newcommand{\Hbb}{\mbb{H}}

\newcommand{\Zc}{\mcl{Z}}
\newcommand{\ggt}{\mfr{g}}

\newcommand{\suml}[2]{\sum\limits_{{#1}}^{{#2}}}

\newcommand{\Ga}{\Gamma}

\newcommand{\la}{\lambda}

\newcommand{\ph}{\varphi}

\newcommand{\GL}{\mbf{GL}}

\newcommand{\Or}{\mbf{O}}
\newcommand{\SO}{\mbf{SO}}

\newcommand{\SU}{\mbf{SU}}

\newcommand{\sog}{\mfr{so}}

\newcommand{\sug}{\mfr{su}}

\DeclareMathOperator{\Lie}{Lie}
\DeclareMathOperator{\Ker}{Ker}

\DeclareMathOperator{\Ad}{Ad}

\DeclareMathOperator{\Aut}{Aut}
\DeclareMathOperator{\In}{In}

\DeclareMathOperator{\rk}{rk}

\DeclareMathOperator{\Spec}{Spec}

\newenvironment{nums}[1]{\renewcommand{\no}{#1}\begin{enumerate}}{\end{enumerate}}

\renewcommand{\ge}{\geqslant}
\renewcommand{\le}{\leqslant}

\newcommand{\bes}{\infty}
\newcommand{\es}{\varnothing}

\newcommand{\subs}{\subset}
\newcommand{\sups}{\supset}

\newcommand{\sm}{\setminus}

\newcommand{\cln}{\colon}

\newcommand{\Lra}{\Leftrightarrow}

\newcommand{\hra}{\hookrightarrow}
\newcommand{\dv}{\smash{\mskip3mu\lower1pt\hbox{\vdots}\mskip3mu}}

\newcommand{\thra}{\twoheadrightarrow}

\renewcommand{\refname}{References}

\begin{document}

\maketitle

The article is devoted to the question whether the orbit space of a~compact linear group is a~topological manifold and a~homological manifold. In the paper, the case of a~simple three-dimensional group is considered. An upper
bound is obtained for the sum of the half-dimension integral parts of the irreducible components of a~representation whose quotient space is a~homological manifold, that enhances an earlier result giving the same bound if the
quotient space of a~representation is a~smooth manifold. The most of the representations satisfying this bound are also researched before. In the proofs, standard arguments from linear algebra, theory of Lie groups and algebras
and their representations are used.

\smallskip

\textit{Key words}\:
Lie group, linear representation of a~group, topological quotient space of an action, topological manifold, homological manifold.

\section{Introduction}\label{introd}

Consider a~faithful linear representation of a~compact Lie group~$G$ in a~real vector space~$V$. The question is, whether the
topological quotient space $V/G$ of this action is a~topological manifold, and whether it is a~homological manifold. Further, for brevity, we will say
simply <<manifold>> for a~topological manifold.

Without loss of generality, we can suppose that $V$ is a~Euclidian space, $G$ is a~Lie subgroup of the group $\Or(V)$, and the representation $G\cln V$ is tautological.

This problem was researched in~\cite{MAMich,Lange} for finite groups. Besides, the author's papers \cite{My1,My2,My3,My4} study both topological and
differential-geometric properties of the quotient space for different group classes\: for groups with commutative connected component~\cite{My1} and for simple groups
of classical types~\cite{My2,My3,My4}. The author's articles \cite{homoarx,homo,homo1} also consider groups with commutative connected component and strengthen the
<<topological>> part of the results of~\cite{My1}. This paper gives a~similar strengthening of the results of~\cite{My2} for simple three-dimensional groups.

Denote by~$G^0$ the connected component of the group~$G$ and by~$\ggt$ its Lie algebra.

Suppose that $\ggt\cong\sug_2$\ti or, equivalently, that the group~$G^0$ is isomorphic to one of the groups $\SU_2$ and $\SO_3$.

Let $n_1\sco n_L$ be the dimensions of irreducible components of the representation $\ggt\cln V$ with considering multiplicities and in descending
order. Since the representation $G\cln V$ is faithful, we have $n_1\sge n_l>1=n_{l+1}=\dots=n_L$, where $l\in\{1\sco N\}$. Denote by $q(V)$ the number
$\suml{i=1}{L}\bs{\frac{n_i}{2}}=\suml{i=1}{l}\bs{\frac{n_i}{2}}\in\N$.

The main result of the article is the following theorem.

\begin{theorem}\label{main} If $\ggt\cong\sug_2$ and $V/G$ is a~homological manifold, then $q(V)\le4$.
\end{theorem}

\section{Auxiliary facts}\label{facts}

This section contains a~number of auxiliary notations and statements, including the ones taken from the articles cited (all new statements are with proofs).

\begin{lemma}\label{prop} Let $X$ be a~topological space and $n$ a~positive integer.
\begin{nums}{-1}
\item If $X$ is a~simply connected homological $n$\dh sphere, then $X\cong S^n$.
\item The cone over the space~$X$ is a~homological $(n+1)$\dh manifold if and only if $X$ is a~homological $n$\dh sphere.
\end{nums}
\end{lemma}

\begin{proof} See Theorem~2.3 and Lemma~2.6 in~\cite[\Ss2]{Lange}.
\end{proof}

Traditionally, denote by~$\T$ the Lie group $\bc{\la\in\Cbb\cln|\la|=1}$ by multiplication.

Suppose that we have a~Euclidian space~$V$ and a~compact group $G\bw\subs\Or(V)$ with Lie algebra $\ggt\subs\sog(V)$. Consider an arbitrary
vector $v\in V$. The subspaces $\ggt v$ and $N_v:=(\ggt v)^{\perp}$ of the space~$V$ are invariant under the stationary subgroup~$G_v$ of the vector~$v$.
The stationary subalgebra~$\ggt_v$ of the vector~$v$ coincides with $\Lie G_v$. Set $M_v:=N_v\cap(N_v^{G_v})^{\perp}\subs N_v$. Clearly,
$N_v=N_v^{G_v}\oplus M_v\bw\subs V$ and $G_vM_v=M_v$.

\begin{stm}\label{Mv} In each $G^0$\dh invariant subspace $V'\subs V$, there exists a~vector~$v$ such that $M_v\subs(V')^{\perp}$.
\end{stm}

\begin{proof} See Statement~2.2 in~\cite[\Ss2]{My1}.
\end{proof}

\begin{theorem}\label{slice} Let $v\in V$ be some vector. If $V/G$ is a~homological manifold, then $N_v/G_v$ and $M_v/G_v$ are homological
manifolds.
\end{theorem}

\begin{proof} See Theorem~4 and Corollary~5 in~\cite{homo}.
\end{proof}

\begin{df} A~linear operator in a~space over some field is called \textit{a~reflection} (resp. \textit{a~pseudoreflection}) if the subspace of its
stable points has codimension $1$ (resp.~$2$).
\end{df}

\begin{df} Let $K$ be the Lie group $\bc{v\in\Hbb\cln\hn{v}=1}$ by multiplication and $\Ga\subs K$ the inverse image of the dodecahedral rotation group by the covering
homomorphism $K\thra\SO_3$. \textit{The Poincare group} is defined as the linear group obtained by the restriction of the action $K\cln\Hbb$ with left shifts onto the
subgroup $\Ga\subs K$.
\end{df}

\begin{theorem}\label{lang} If the group $G\subs\Or(V)$ is finite and $V/G$ is a~homological manifold, then there are decompositions $G=G_0\times G_1\sti G_k$
and $V=V_0\oplus V_1\sop V_k$ \ter{$k\in\Z_{\ge0}$} such that
\begin{itemize}
\item the subspaces $V_0,V_1\sco V_k\subs V$ are pairwise orthogonal and $G$\dh invariant\~
\item for each $i,j=0\sco k$, the linear group $(G_i)|_{V_j}\subs\Or(V_j)$ is trivial if $i\ne j$, generated by pseudoreflections if $i=j=0$, and isomorphic to
the Poincare group if $i=j>0$ \ter{in part, $\dim V_j=4$ for any $j=1\sco k$}.
\end{itemize}
\end{theorem}

\begin{proof} See Proposition~3.13 in~\cite[\Ss3]{Lange}.
\end{proof}

On the space~$\ggt$, there is an $\Ad(G)$\dh invariant scalar product, according to which, we will further identify the spaces $\ggt$ and~$\ggt^*$. If $\ggt'\subs\ggt$
is a~one-dimensional subalgebra and $V'\subs V$ is a~subspace, then $\ggt'V'\subs V$ is a~subspace of dimension no greater than $\dim V'$.

Recall the definitions of \textit{$q$\dh stable} ($q\in\N$) and \textit{indecomposable} sets of vectors of finite-dimensional spaces over fields~\cite[\Ss1]{My1},
that are also needed in this article.

A~decomposition of a~vector set of a~finite-dimensional linear space onto components will be defined as its representation as the union of its subsets whose
linear spans are linearly independent. If at least two of these linear spans are nontrivial, then such a~decomposition will be called \textit{proper}.
Say that a~vector set is \textit{indecomposable} if it does not admit any proper decomposition onto components. Each vector set can be decomposed onto
indecomposable components uniquely (up to the zero vector distribution), and for any of its decomposition onto components, each component is the union of
some of its indecomposable components (again up to the zero vector).

\begin{df} A~finite vector set of a~finite-dimensional space, with considering multiplicities, will be called \textit{$q$\dh stable} ($q\in\N$), if its
linear span is preserved after deleting of any no greater than $q$ vectors (again taking multiplicities into account).
\end{df}

For an arbitrary finite vector set~$P$ in a~finite-dimensional space over some field, with considering multiplicities, the number of nonzero vectors of the
set~$P$ (again considering multiplicities) will be denoted by~$\hn{P}$.

Assume that the group~$G^0$ is commutative, i.\,e. a~torus.

Each irreducible representation of the group~$G^0$ is one- or two-dimensional. Recall the concept of the weight of its irreducible representation given
in~\cite[\Ss1]{My1}.

An arbitrary two-dimensional irreducible representation of the group~$G^0$ has a~$G^0$\dh invariant complex structure, and we can consider it as
a~one-dimensional complex representation of the group~$G^0$, naturally matching it with a~weight\ti a~Lie group homomorphism $\la\cln G^0\to\T$\ti and identifying
the latter with its differential\ti the vector $\la\in\ggt^*$. Match a~one-dimensional representation of the group~$G^0$ with the weight $\la:=0\in\ggt^*$.

Classes of isomorphic irreducible representations of~$G^0$ are characterized by weights $\la\in\ggt^*=\ggt$ defined up to multiplying by $(-1)$.

Let $P\subs\ggt$ be the set of weights $\la\in\ggt$ corresponding to decomposing the representation $G^0\cln V$ into the direct sum of irreducible ones (with considering
multiplicities). The set $P\subs\ggt$ does not depend on choosing this decomposition (up to multiplying the weights by $(-1)$). Since the representation $G\cln V$
is faithful, we have $\ha{P}=\ggt$.

{\newcommand{\refr}{see \cite[Theorem~4]{homo1}}
\begin{theorem}[\refr]\label{submain} Suppose that $V/G$ is a~homological manifold and $P\subs\ggt$ is a~$2$\dh stable set. Then the representation
$G\cln V$ is the direct product of representations $G_l\cln V_l$ \ter{$l=0\sco p$} such that
\begin{nums}{-1}
\item for any $l=0\sco p$, the quotient space $V_l/G_l$ is a~homological manifold\~
\item $|G_0|<\bes$\~
\item for each $l=1\sco p$, the group~$G_l$ is infinite and the weight set of the representation $G_l\cln V_l$ is indecomposable, $2$\dh stable and does not contain
zeroes.
\end{nums}
\end{theorem}}

\begin{theorem}\label{main1} Assume that $\dim G=1$ and the set $P\subs\ggt$ is $2$\dh stable and does not contain zeroes. If $V/G$ is a~homological manifold, then
$\hn{P}=3$.
\end{theorem}

\begin{proof} Since $0\notin P$, the space~$V$ has a~$G^0$\dh invariant complex structure. If the group $G\subs\Or(V)$ does not contain complex reflections, then
the statement follows from Theorem~6 of the paper~\cite{homo1}. As for the arbitrary case, it can be reduced (see~\cite[\Sss3,\,7]{My1}) to that of a~representation
of a~one-dimensional group without complex reflections whose weight set is obtained form~$P$ with multiplying all weights by nonzero scalars.
\end{proof}

\begin{imp}\label{p3} Suppose that $\dim G=1$ and $P\subs\ggt$ is a~$2$\dh stable set. If $V/G$ is a~homological manifold, then $\hn{P}=3$.
\end{imp}

\begin{proof} Follows from Theorems \ref{submain} and~\ref{main1}.
\end{proof}

\begin{imp}\label{ple3} If $\dim G=1$ and $V/G$ is a~homological manifold, then $\hn{P}\bw\le3$ \ter{or, equivalently, $\dim(\ggt V)\le6$}.
\end{imp}

\begin{proof} Assume that $\hn{P}>3$. Then $P\subs\ggt$ is a~$2$\dh stable set. By Corollary~\ref{p3}, $\hn{P}=3$. So, we came to a~contradiction.
\end{proof}

\section{Proofs of the results}\label{prove}

In this section, Theorem~\ref{main} is proved.

For the rest of the paper, we will assume that $\ggt\cong\sug_2$, i.\,e. that the group~$G^0$ is isomorphic to $\SU_2$ or $\SO_3$.
Set $V_0:=V^{G^0}\subs V$. In terms of~\S\,\ref{introd}, $L-l\bw=\dim V_0$, $V_0\ne V$, and the numbers $n_1\sco n_l$ are the dimensions of irreducible components of the
representation $\ggt\cln V_0^{\perp}$ (with considering multiplicities), each of them being either divisible by~$4$ or odd. If $\ggt'\bw\subs\ggt$ is a~proper subalgebra,
then $\dim\ggt'=1$, and the subspace $\ggt'V\subs V$ has dimension $2q(V)$. Therefore, $2q(V)\le\dim V$.

It is enough to prove the theorem in the case $V_0=0$ (i.\,e. if the representation $G^0\cln V$ does not have one-dimensional irreducible components). Indeed, there
exists a~vector $v\in V_0$ such that $M_v\subs V_0^{\perp}$ (see Statement~\ref{Mv}). We have $\ggt v=0$, $N_v=V$, $G_v\sups G^0$, $(G_v)^0\bw=G^0$,
$\ggt_v=\ggt$. Further, $M_v=(V^{G_v})^{\perp}\bw\sups(V^{G^0})^{\perp}\bw=V_0^{\perp}\sups M_v$, and, hence, $M_v\bw=V_0^{\perp}$. By Theorem~\ref{slice},
if $V/G$ is a~homological manifold, then $M_v/G_v$ is a~homological manifold. As for the decompositions onto irreducible components of the representations
of the group $G^0\bw=(G_v)^0$ in the spaces $V$ and $M_v$, the latter is obtained from the former by deleting all one-dimensional components. Thus, $q(V)\bw=q(M_v)$.

Further, let us assume that $V/G$ is a~homological manifold and $V_0=0$. We should prove that $q(V)\le4$.

Suppose that there exists a~vector $v\in V$ such that $\dim G_v=1$. Then $\dim\ggt_v=1$, $\dim(\ggt v)=2$. Besides, by Theorem~\ref{slice}, $N_v/G_v$ is a~homological
manifold. According to Corollary~\ref{ple3}, $\dim(\ggt_v N_v)\le6$, $2q(V)\bw=\dim(\ggt_v V)\bw\le\dim(\ggt_v N_v)+\dim(\ggt v)\le8$, $q(V)\le4$.

Further, we will assume that the space~$V$ does not contain a~vector with one-dimensional stationary subgroup and that $q(V)>4$. Consequently,
\begin{itemize}
\item $\dim V\ge2q(V)>8$\~
\item $G^0\cong\SU_2$\~
\item any irreducible component of the representation $G^0\cln V$ has dimension divisible by~$4$\~ the same can be said about each of its subrepresentations\~
\item $(\Ker\Ad)\cap G^0=\Zc(G^0)=\{\pm E\}\subs\Or(V)$.
\end{itemize}

Arbitrary operators $g\in G$ and $\xi\in\ggt^{\Ad(g)}$ in the space~$V$ commute. Hence, for any $g\in G$, the subspaces $V^g,(E-g)V\subs V$
are $(\ggt^{\Ad(g)})$\dh invariant\~ once $\Ad(g)=E$, they are $\ggt$\dh invariant and, thus, $\rk(E-g)\dv4$.

Consider an arbitrary vector $v\in V$. If the subalgebra $\ggt_v\subs\ggt$ is proper, then $\dim\ggt_v=1$ that contradicts the assumption. So, $\ggt_v=\ggt$ or $\ggt_v=0$.
In the former case, we have $G_v\sups G^0$, $v\in V^{G^0}=V_0=0$. Therefore, if $v\ne0$, then $\ggt_v=0$, $|G_v|<\bes$, the map $\ggt\to(\ggt v),\,\xi\to(\xi v)$ is a~linear
isomorphism, each $g\in G_v$ and $\xi\in\ggt$ satisfying $g(\xi v)=\br{\Ad(g)\xi}v$ that implies $(\xi v\in V^g)\Lra(\xi\in\ggt^{\Ad(g)})$. Consequently, if $v\ne0$ and
$g\in G_v$, then $(\ggt v)^g=(\ggt^{\Ad(g)})v$.

For arbitrary $g\in G$ and $\xi\in\ggt$, denote by~$\ph_{g,\xi}$ the linear map of the space~$V$ to the outer direct sum of two copies of the space $(E-g)V$ defined by the
rule $v\to\br{(E-g)v,(E-g)\xi v}$.

\begin{lemma}\label{inj} For any $g\in G$ and $\xi\in\ggt\sm(\ggt^{\Ad(g)})$, we have $\Ker\ph_{g,\xi}=0$.
\end{lemma}

\begin{proof} If $v\ne0$ and $v\in\Ker\ph_{g,\xi}$, then $(E-g)v=(E-g)\xi v=0$, i.\,e. $g\in G_v$ and $\xi v\in V^g$ following $\xi\in\ggt^{\Ad(g)}$ that contradicts the
condition.
\end{proof}

\begin{imp} If $g\in G$ and $\Ad(g)\ne E$, then $\dim V\le2\cdot\rk(E-g)$.
\end{imp}

Our nearest goal is proving the next theorem.

\begin{theorem}\label{stab} For each $v\in V\sm\{0\}$, we have $[G_v,G_v]=G_v\subs\Ker\Ad$.
\end{theorem}

For proving Theorem~\ref{stab}, fix an arbitrary vector $v\bw\in V\sm\{0\}$.

Denote by~$\pi$ the homomorphism $G_v\to\Or(N_v),\,g\to g|_{N_v}$ and by~$H_v$ the subgroup $\pi(G_v)\subs\Or(N_v)$. By Theorem~\ref{slice}, $N_v/G_v$ is a~homological
manifold. Besides, $|G_v|<\bes$. According to Theorem~\ref{lang}, there are decompositions $H_v\bw=H_0\times H_1\sti H_k$ and $N_v=W_0\oplus W_1\sop W_k$ \ter{$k\in\Z_{\ge0}$}
such that
\begin{itemize}
\item the subspaces $W_0,W_1\sco W_k\subs N_v$ are pairwise orthogonal and $G_v$\dh invariant\~
\item for each $i,j=0\sco k$, the linear group $(H_i)|_{W_j}\subs\Or(W_j)$ is trivial if $i\ne j$, generated by pseudoreflections if $i=j=0$, and isomorphic to the
Poincare group if $i=j>0$ \ter{in part, $\dim W_j=4$ for any $j=1\sco k$}.
\end{itemize}
It is well known that the Poincare group coincides with its commutant\~ the same can be said about each of the groups $H_i$, $i=1\sco k$.

If $g\in G_v$, then $\rk(E-g)-\dim\br{(E-g)N_v}=\dim\br{(E-g)(\ggt v)}\bw=\rk\br{E\bw-\Ad(g)}\bw\le2$\~ once $\Ad(g)=E$, we have $\rk(E-g)=\dim\br{(E-g)N_v}$.

\begin{lemma}\label{triv} If $g\in G_v$ and $\dim\br{(E-g)N_v}\le2$, then $g=E$.
\end{lemma}

\begin{proof} By condition, $\rk(E-g)\le4$. If $\Ad(g)\ne E$, then $\dim V\bw\le2\bw\cdot\rk(E\bw-g)\le8$ while $\dim V>8$. Hence, $\Ad(g)=E$ implying, firstly,
$\rk(E-g)\dv4$ and, secondly, $\rk(E-g)=\dim\br{(E-g)N_v}\le2$. Thus, $\rk(E-g)=0$, $g=E$.
\end{proof}

According to Lemma~\ref{triv}, $\Ker\pi=\{E\}\subs G_v$, i.\,e. $\pi$ is an isomorphism $G_v\bw\to H_v$. Therefore, setting $G_i:=\pi^{-1}(H_i)\subs G_v$
\ter{$i=0\sco k$}, we obtain that
\begin{itemize}
\item $G_v=G_0\times G_1\sti G_k$\~
\item the group~$G_0$ is generated by the elements $g\in G_v$ such that $\dim\br{(E\bw-g)N_v}\le2$ (and, by Lemma~\ref{triv}, is trivial)\~
\item each of the groups $G_i$, $i=1\sco k$, coincides with its commutant\~
\item if $i\in\{1\sco k\}$ and $g\in G_i\sm\{E\}$, then $N_v^g=N_v\cap W_i^{\perp}$ and $(E-g)N_v=W_i$ (consequently, $\dim\br{(E-g)N_v}=4$, $\rk(E-g)\le6$).
\end{itemize}
It follows from said above that $G_v=G_1\sti G_k=[G_v,G_v]$.

\begin{lemma} Each of the groups $\Ad(G_i)$, $i=1\sco k$, is commutative.
\end{lemma}

\begin{proof} Suppose that there exist a~number $i\in\{1\sco k\}$ and elements $g,h\in G_i$ such that the operators $\Ad(g)$ and $\Ad(h)$ do not commute.

We have $\Ad(g),\Ad(h)\ne E$, the subspaces $\ggt^{\Ad(g)},\ggt^{\Ad(h)}\subs\ggt$ being different and one-dimensional. Hence, $\ggt^{\Ad(h)}=\R\xi$
($\xi\in(\ggt^{\Ad(h)})\sm(\ggt^{\Ad(g)})$), thus, $\xi V^h\bw\subs V^h$ and, also, $(\ggt v)^h=(\ggt^{\Ad(h)})v=\R(\xi v)$. Further, $g,h\in G_i\sm\{E\}$
following, firstly, $\rk(E-h)\le6$ and, secondly, $N_v^g=N_v^h=N_v\cap W_i^{\perp}$, $V^h=N_v^g\oplus\br{\R(\xi v)}$,
$(E-g)\xi V^h\subs(E-g)V^h=\R\br{(E-g)\xi v}$, $\dim(\ph_{g,\xi}V^h)\le2$. According to Lemma~\ref{inj}, $\Ker\ph_{g,\xi}=0$ implying $\dim V^h=\dim(\ph_{g,\xi}V^h)\le2$ and,
consequently, $\dim V=\rk(E-h)+\dim V^h\le8$ while $\dim V>8$. The contradiction obtained completes the proof.
\end{proof}

For each $i=1\sco k$, we have $\Ad(G_i)\bw=\Ad\br{[G_i,G_i]}\bw=\bs{\Ad(G_i),\Ad(G_i)}\bw=\{E\}$, i.\,e. $G_i\subs\Ker\Ad$. Therefore, $G_v=G_1\sti G_k\subs\Ker\Ad$.

So, we completely proved Theorem~\ref{stab}.

\begin{imp}\label{strv} For any $v\in V\sm\{0\}$, we have $G_v\cap G^0=\{E\}$.
\end{imp}

\begin{proof} By Theorem~\ref{stab}, $G_v\subs\Ker\Ad$, $G_v\cap G^0\subs(\Ker\Ad)\cap G^0=\{\pm E\}\bw\subs\Or(V)$.
\end{proof}

There exists an embedding $\T\hra G^0$, so, the group~$\T$ can be identified with its image by this embedding and considered as a~subgroup of the group~$G^0$.
According to Corollary~\ref{strv}, each irreducible subrepresentation of the representation $\T\cln V$ is faithful and, hence, isomorphic to the representation
$\T\cln\Cbb$ by multiplying. Thus, the space~$V$ is supplied with a~complex structure in whose terms the action $\T\cln V$ is proceeded with
multiplying by scalars. All operators of the group $\Ker\Ad$ commute with all operators of the group $G^0\sups\T$, and, also, $(\Ker\Ad)\cap G^0=\{\pm1\}\subs\T$.
Therefore, $\Ker\Ad$ is a~finite subgroup of the group $\GL_{\Cbb}(V)$, each of its operators~$g$ is semisimple over the field~$\Cbb$ and satisfies the relation
$(\Spec_{\Cbb}g)\subs\T\subs G^0$.

In the group~$G$, denote by~$H$ the subgroup generated by the union of all subgroups $G_v$, $v\in V\sm\{0\}$. By Theorem~\ref{stab}, $H\subs\Ker\Ad$.

\begin{prop}\label{gla} For any $g\in\Ker\Ad$, we have $(\Spec_{\Cbb}g)\subs(gH)\cap\T$.
\end{prop}

\begin{proof} Take an arbitrary element $\la\in(\Spec_{\Cbb}g)$. There exists a~vector $v\bw\in V\sm\{0\}$ such that $gv=\la v$, so, $\la\in gG_v\subs gH$.
\end{proof}

\begin{lemma} For each $g\in\Ker\Ad$, we have $g^2=E$.
\end{lemma}

\begin{proof} As said above, $H\subs\Ker\Ad$. So, $gH\subs g(\Ker\Ad)=\Ker\Ad$. By Proposition~\ref{gla},
$(\Spec_{\Cbb}g)\subs(gH)\cap\T\subs(\Ker\Ad)\cap\T=\{\pm1\}\subs\T$.
\end{proof}

\begin{imp}\label{coker} The group $\Ker\Ad$ is commutative.
\end{imp}

\begin{imp}\label{stri} For any $v\in V\sm\{0\}$, we have $G_v=\{E\}$.
\end{imp}

\begin{proof} Follows from Theorem~\ref{stab} and Corollary~\ref{coker}.
\end{proof}

\begin{imp}\label{htri} The subgroup $H\subs G$ is trivial.
\end{imp}

\begin{lemma}\label{ker0} We have $\Ker\Ad\subs G^0$.
\end{lemma}

\begin{proof} Let $g\in\Ker\Ad$ be any element. It follows from Proposition~\ref{gla} and Corollary~\ref{htri} that
$(\Spec_{\Cbb}g)\subs\{g\}\cap\T$. On the other hand, $(\Spec_{\Cbb}g)\ne\es$, so, $g\in\T\subs G^0$.
\end{proof}

Since $\Aut(\ggt)=\In(\ggt)$, we have $\Ad(G)=\Ad(G^0)$ implying (with Lemma~\ref{ker0}) $G=G^0(\Ker\Ad)=G^0\cong\SU_2$. Thus,
$\pi_3(G)\bw\cong\pi_3(\SU_2)\cong\pi_3(S^3)\cong\Z$.

Let $S\subs V$ be the unit sphere and $M$ the quotient space $S/G$.

We have $\dim V>8$, $\dim S>7$\~ hence, $S$ and~$M$ are connected topological spaces and, also, $\pi_k(S)=\{e\}$ \ter{$k=1\sco7$}. According to
Corollary~\ref{stri}, the action $G\cln S$ is free. The factorization map $S\thra M$ is a~locally product bundle with fibre~$G$. The corresponding
exact homotopic sequence gives the relations $\pi_k(M)\bw\cong\pi_{k-1}(G)$ \ter{$k=2\sco7$} and $\pi_1(M)\cong G/G^0=\{e\}$. By Lemma~\ref{prop}, $M\cong S^m$
\ter{$m:=\dim S-3>4$}\~ on the other hand, $\pi_4(M)\cong\pi_3(G)\cong\Z$. So, we came to a~contradiction that completely proves Theorem~\ref{main}.

\newpage

{\renewcommand{\refname}{References}
}

\end{document}